\documentclass[a4paper]{article}
\usepackage{amssymb}
\usepackage{graphicx}
\usepackage{amsmath}
\usepackage{amsfonts}
\usepackage{amsthm}
\usepackage{mathrsfs}
\usepackage{dsfont}
\usepackage{indentfirst}
\usepackage{esint}
\usepackage{hyperref}
\usepackage{yhmath}

\numberwithin{equation}{section}

\newcommand{\diam}{\mathrm{diam}}

\title{\Large \bf \boldmath\ \\ Polynomial Growth Harmonic Functions on Groups of Polynomial Volume Growth} 

\author{\large  Bobo Hua$^\ast$\ \ \ \  J\"{u}rgen Jost$^{\dag}$} 

\date{}

\par
\begin{document}

\maketitle

\renewcommand{\thefootnote}{\fnsymbol{footnote}}

\footnotetext{\hspace*{-5mm} \begin{tabular}{@{}r@{}p{13.4cm}@{}}
$^\ast$ $^{\dag}$ & Max Planck Institute for Mathematics in the Sciences, Leipzig, 04103, Germany.\\
$^{\dag}$ & supported by ERC Advanced Grant FP7-267087.\\
&{Email: bobohua@mis.mpg.de, jost@mis.mpg.de}\\
&The Mathematics Subject Classification 2010: 53C21, 20F65, 20F69,
31C05, 05C63, 82B41.
\end{tabular}}

\renewcommand{\thefootnote}{\arabic{footnote}}

\newtheorem{theorem}{Theorem}[section]
\newtheorem{con}[theorem]{Conjecture}
\newtheorem{lemma}[theorem]{Lemma}
\newtheorem{corollary}[theorem]{Corollary}
\newtheorem{definition}[theorem]{Definition}
\newtheorem{conv}[theorem]{Assumption}
\newtheorem{remark}[theorem]{Remark}

\bigskip

\begin{abstract}  We consider  harmonic functions of  polynomial growth
  of some order $d$ on Cayley graphs of groups of polynomial volume
  growth of order $D$ w.r.t. the word metric and  prove
  the optimal  estimate for the dimension of the space of such
  harmonic functions. More precisely, the dimension of this space of
  harmonic functions is at most of
  order $d^{D-1}$.  As in the already known  Riemannian case, this
  estimate  is polynomial in the growth degree. More generally, our
  techniques also apply to graphs roughly isometric to Cayley graphs of groups of polynomial volume growth.

\end{abstract}
\section{Introduction}
In \cite{G}, Gromov proved the celebrated structure theorem for
finitely generated groups of polynomial volume growth. His result
says that every such group is virtually nilpotent, i.e. it has a
nilpotent subgroup of finite index. Kleiner \cite{Kl} proved the
polynomial growth harmonic function theorem on Cayley graphs of such
groups. Namely, the space of polynomial growth harmonic functions
with a fixed growth rate is of finite dimension. With this result,
he obtained a new proof of Gromov's theorem. Then Shalom and Tao
\cite{ST} gave a quantitative version of Kleiner's result. Since
they used Colding and Minicozzi's original argument in \cite{CM2},
the dimension estimate of polynomial growth harmonic functions they
obtained is exponential in the growth degree. In this paper, we use
the more delicate volume growth property by Pansu \cite{P} to obtain
the optimal dimension estimate analogous to the Riemannian case on
Cayley graphs of groups of polynomial volume growth. This optimal
estimate is polynomial in the growth degree.

The history leading to these results started in Riemannian geometry.
In 1975, Yau \cite{Y1} proved the Liouville theorem for harmonic
functions on Riemannian manifolds with nonnegative Ricci
curvature. Then Cheng-Yau \cite{CheY} used Bochner's technique to
derive a gradient estimate for positive harmonic functions  which implies that sublinear growth harmonic
functions on these manifolds are constant. Then Yau \cite{Y2,Y3}
conjectured that the space of polynomial growth harmonic functions
with growth rate less than or equal to $d$ on Riemannian manifolds
with nonnegative Ricci curvature is of finite dimension. Li-Tam
\cite{LT} and Donnelly-Fefferman \cite{DF} independently solved the
conjecture for manifolds of dimension two. Then Colding-Minicozzi
\cite{CM1,CM2,CM3} proved Yau's conjecture in any dimension by using the volume
doubling property and the Poincar\'e inequality for arbitrary
dimension. A simplified argument by the mean value inequality can be
found in \cite{L1,CM4} where the dimension estimate is asymptotically
optimal. This inspired many generalizations on manifolds
\cite{W,T,STW,LW1,LW2,CW,KLe,Le}. Essentially, the crucial ingredients
of these proofs are the volume growth property and the Poincar\'e inequality (or mean value inequality).

It was then found that this line of reasoning carries over to
graphs. Let $(G,S)$ be a Cayley graph of a group $G$ with a finite
generating set $S$. There is a natural metric on $(G,S)$ called the
word metric, denoted by $d^S$. Let $B_p^S(n):=\{x\in G| d^S(x,p)\leq
n\}$ denote the closed geodesic ball of radius $n\in \mathds{N}$
centered at $p\in G.$ We denote by $|B_p^S(n)|:=\#\{B_p^S(n)\}$ the
volume (cardinality) of the set $B_p^S(n)$. When $e$ is the unit
element of $G$, the volume  $\beta_S(n):=|B_{e}^S(n)|$  of
$B_e^S(n)$ is called the growth function of the group. The intensive
investigation of the growth function of a finitely generated group
began after Milnor's work \cite{M1} (see \cite{M2,Wl} or the survey
papers \cite{Gri1,Gri2,Gri3}), although the notion of the growth of
a finitely generated had already been introduced earlier by A.
\c{S}varc (A. Schwarz) \cite{Sv}. A group $G$ is called of
polynomial growth if $\beta_S(n)\leq C n^D,$ for any $n\geq1$ and
some $D>0,$ which is independent of choice of the generating set $S$
since $d^S$ and $d^{S_1}$ are bi-Lipschitz equivalent for any two
finite generating sets $S$ and $S_1.$ By Gromov's theorem and Bass'
volume growth estimate of nilpotent groups \cite{Ba}, for any group
$G$ of polynomial volume growth there are constants $C_1(S), C_2(S)$
depending on $S$ and $D\in\mathds{N}$ such that for any $n\geq1,$
\begin{equation}\label{BAV1}C_1 n^D\leq \beta_S(n)\leq C_2
  n^D,\end{equation}
where the integer $D$ is called the homogeneous dimension or the
growth degree of $G$. In this paper, since $D$ is a dimensional
constant of $G,$ we always omit the dependence of $D$ for constants.
Then it is easy to see that for any $n\geq1, p\in G$
\begin{equation}\label{VDP}|B_p^S(2n)|\leq C(S)
  |B_p^S(n)|,\end{equation} which is called the volume doubling property.
Moreover, by the group structure, Kleiner \cite{Kl} obtained the
Poincar\'e inequality for Cayley graphs of finitely generated
groups. This inequality says that for finitely generated groups of polynomial volume growth, by \eqref{VDP}, the uniform Poincar\'e inequality holds
\begin{equation}\label{PI}\sum_{x\in B_p^S(n)}(u(x)-\bar{u})^2\leq C(S)n^2\sum_{x,y\in B_p^S(3n);x\sim y}(u(x)-u(y))^2,\end{equation} where $\bar{u}=\frac{1}{|B_p^S(n)|}\sum_{x\in B_p^S(n)}u(x)$ and $x\sim y$ means they are neighbors.

The discrete Laplacian operator $L^S$ on the Cayley graph $(G,S)$ is defined as
\begin{equation}
L^Su(x)=\sum_{y\sim x}(u(y)-u(x)) \text{ for }x\in G.
\end{equation}
A function $u:G\rightarrow \mathds{R}$ is called harmonic on $G$ if $L^Su(x)=0$
for any $x\in G.$ For some fixed $ p\in G,$ let us denote by
\begin{equation}
H^d(G,S):=\{u:G\rightarrow \mathds{R}\mid\ L^S u=0, |u|(x)\leq
C(d^S(p,x)+1)^d\}
\end{equation}
the space of polynomial growth harmonic functions of growth rate less
than or equal to $d$ on $(G,S).$ Kleiner \cite{Kl} (or \cite{ST})
adapted the original argument in Colding-Minicozzi \cite{CM2} to show
that $\dim H^d(G,S)\leq C_1(S)e^{C_2(S)d^2},$ i.e. it is exponential
in the growth degree which is not optimal. Delmotte \cite{D1} proved
the polynomial growth harmonic function theorem $\dim H^d(G,S)\leq
C(S)d^{v(S)}$ for graphs satisfying the volume doubling property
\eqref{VDP} and the Poincar\'e inequality \eqref{PI}. Instead of using
the volume doubling property \eqref{VDP}, as in the Riemannian
manifold case (see \cite{CM3,CM4,L1}), Hua-Jost-Liu \cite{HJL} and
Hua-Jost \cite{HJ} applied a more delicate volume growth property called the relative volume comparison to obtain the optimal dimension estimate on planar graphs with nonnegative combinatorial curvature.

Besides Bass' volume growth estimate \eqref{BAV1}, Pansu \cite{P}
proved a more delicate volume growth property for groups of
polynomial volume growth (see also Breuillard \cite{Bre}). He showed
that for the Cayley graph $(G,S)$ of a group $G$ of polynomial
volume growth the following limit exists
\begin{equation}
\lim_{n\rightarrow \infty} \frac{\beta_S(n)}{n^D}=C(S),
\end{equation} where $D\in \mathds{N}$. Then it is easy to see (Lemma \ref{RVCAI}) that for any $\theta\ll1,$ there exists $R_0(\theta,S)$ such that
\begin{equation}\label{RVII1}
\frac{|B_p(R)|}{|B_p(r)|}\leq (1+\theta)\left(\frac{R}{r}\right)^D,
\end{equation} for any $R\geq r\geq R_0(\theta,S),$ which is called the relative volume comparison in the large. By combining this with the mean value inequality (see Lemma \ref{MVI1}), which is a result of Moser iteration (c.f. \cite{D2,HS,CoG}), we obtain the optimal dimension estimate by Li's argument \cite{L1,L2}. This estimate is asymptotically optimal since it is achieved in the finitely generated abelian group case (see \cite{HJLi,H,N}).

\begin{theorem}\label{MTAL}
Let $(G,S)$ be a Cayley graph of a group of polynomial volume growth with homogeneous dimension $D$. Then for $d\geq1$
$$\dim H^d(G,S)\leq C(S)d^{D-1}.$$
\end{theorem}

Note that Alexopoulos \cite{Al1} characterized polynomial growth harmonic functions on groups of polynomial volume growth. By Gromov's theorem \cite{G}, every finitely generated group $G$ of polynomial growth is virtually nilpotent (i.e. it has a nilpotent subgroup $H$ of finite index). For some torsion-free subgroup $H'$ of $H,$ it can be embedded as a lattice in a simply connected nilpotent Lie group $N.$ By considering the exponential coordinate of $N,$ $N$ is identified with $\mathds{R}^q$ (see \cite{Ra}). Alexopoulos proved that every polynomial growth harmonic function on $G$ is the restriction to $H'$ of some polynomial on $\mathds{R}^q.$ It seems hard to calculate the precise dimension of polynomial growth harmonic functions by this method which depends on the embedding of the nilpotent subgroup into a simply connected Lie group. Instead of doing that, we give a dimension estimate by the geometric analysis methods of Colding-Minicozzi and Li described above.

In the second half of the paper, we generalize our result and prove the optimal dimension
estimate for polynomial growth harmonic functions on graphs with
bounded geometry roughly isometric to Cayley graphs of groups of
polynomial volume growth. Let $X=(V,E)$ be a graph with the natural
metric $d^X.$ The degree of a vertex $x\in V$ is defined as $\deg
x:=\#\{y\in V| y\sim x\}.$ A graph $X$ with bounded geometry means
that $\deg x\leq \Delta$ for some $\Delta>0$ and all $x\in V.$ The
graph $(X,d^X)$ is called roughly isometric to a Cayley graph
$(G,S,d^S)$ of a group if there is mapping $\phi:X\rightarrow G$ which
is (roughly) bi-Lipschitz in the large scale (see Definition
\ref{DEFR}, or \cite{Gro,BBI,Wo}). It is well known \cite{CS,Kl} that
the volume doubling property and the Poincar\'e inequality are
rough-isometry invariants. Let $H^d(X)$ denote the space of polynomial
growth harmonic functions on $X$ with growth rate less than or equal
to $d.$ If $(X,d^X)$ is roughly isometric to a Cayley graph $(G,S)$ of
a group of polynomial volume growth, then Colding-Minicozzi and Li's
arguments yield that $$\dim H^d(X)\leq Cd^D,$$ for $d\geq1$ where $D$
is the homogenous dimension of $G.$ This is not asymptotically
optimal. Since we cannot hope to deduce the relative volume comparison \eqref{RVII1} on $X$ from the rough isometry $\phi,$ it is hard to obtain the optimal dimension estimate on $X.$

Kanai \cite{K} proved a famous theorem that for manifolds or graphs
the property to be parabolic (recurrent in the terminology of random
walks) is a rough-isometry invariant (see also \cite{Wo}). In contrast, a counterexample of Benjamini \cite{Ben} shows the instability of the Liouville theorem under rough isometries. He constructed two roughly isometric graphs: one admits an infinite dimensional space of nonconstant bounded harmonic functions, while the other admits only constant ones. That is, the dimension estimate of polynomial growth harmonic functions is not a rough-isometry invariant in general. In addition, Lee \cite{Le} obtained the optimal dimension estimate for Riemannian manifolds with bounded geometry roughly isometric to Riemannian manifolds with nonnegative Ricci curvature. His method is quite different from ours and seems not suitable for the discrete case (for graphs).

In order  to extend the optimal dimension
estimate for polynomial growth harmonic functions on a space $X$
roughly isometric to a Cayley graph $(G,S)$ of a group of polynomial
volume growth, we use an idea of \cite{HJ}. Since the relative volume comparison \eqref{RVII1}
holds on $(G,S)$ rather than on $X,$ we do the dimension estimate on
$G$ to obtain the dimension estimate on $X.$ For any harmonic function
$u$ on $X,$ we construct a function $\widetilde{u}$ on $G$ (see
\eqref{dfi1}) which preserves the crucial property - the mean value
inequality - even though $\widetilde{u}$ is not necessarily harmonic anymore. We
need to prove the mean value inequality for $\widetilde{u}.$ Firstly,
it is easy to obtain the mean value inequality for a harmonic function
$u$ on $X$ since the volume doubling property and the Poincar\'e
inequality are stable under rough isometries (see Lemma
\ref{MVFXX}). Secondly, from a rough isometry $\phi:X\rightarrow G$ we
can construct an injective rough isometry $\phi': X\rightarrow G'$
where $G'$ is also of polynomial volume growth with the same
homogenous dimension as $G.$ Hence without loss of generality we
assume that the rough isometry $\phi$ is injective. By the injectivity of $\phi$ and the mean value inequality of $u,$ we can prove the mean value inequality in the large for $\widetilde{u}$ (see Theorem \ref{MVIIL}).

Let $P^d(G):=\{u:G\rightarrow \mathds{R}|\ |u|(x)\leq C(d^S(x,p)+1)^d\}$ denote the space of polynomial growth functions on $G$ with growth rate less than or equal to $d.$ It is easy to see that the operator
$$E:H^d(X)\rightarrow P^d(G)$$
$$u\mapsto Eu:=\widetilde{u}$$ is injective and linear. Hence $$\dim
H^d(X)=\dim EH^d(X).$$ It suffices to bound $\dim EH^d(X)$ on $G.$ By
the same argument as in the proof of Theorem \ref{MTAL}, the relative volume comparison \eqref{RVII1} and the mean value inequality in the large \eqref{MVF1} yield
\begin{theorem}\label{MT2}
Let $(X,d^X)$ be a graph with bounded geometry roughly isometric to a Cayley graph $(G,S,d^S)$ of a group of polynomial volume growth of homogenous dimension $D.$ Then for $d\geq1$
$$\dim H^d(X)\leq Cd^{D-1}.$$
\end{theorem}

This line of the proof is even new compared with the results in Lee
\cite{Le}. As an application of Theorem \ref{MT2}, we obtain in
Corollary \ref{COAC} the optimal dimension estimate for polynomial
growth harmonic functions on quasi-transitive graphs of polynomial
volume growth (see \cite{Wo} for the definition).

Our results provide the final answer for the control of
dimension of the space of harmonic functions of some given polynomial
growth on Cayley graphs of groups of polynomial volume growth, and we
can even achieve this control more generally  in rough isometry
classes of such Cayley graphs. Also, as described above, such a
control has been achieved under curvature bounds. We may ask, however,
for which other classes of
graphs such a control might conceivably be possible. In this
direction, in \cite{BDKY}, it was shown that on a percolation cluster
in $\mathbb{Z}^D$, the space of harmonic functions of linear growth,
i.e., $d=1$ in our terminology, is almost surely of dimension at most
$D+1$. From the perspective of our paper, we should expect that also
harmonic functions of polynomial growth for any $d$ can be
controlled.

\section{Preliminaries and Notations}
Let $G$ be a group. It is called finitely generated if it has a finite generating set $S$. We always assume that the generating set $S$ is symmetric, i.e. $S=S^{-1}.$ The Cayley graph of $(G,S)$ is a graph structure $(V,E)$ with the set of vertices $V=G$ and the set of edges $E$ where for any $x,y\in G,$ $xy\in E$ (also denoted by $x\sim y$) if $x=ys$ for some $s\in S.$ The Cayley graph of $(G,S)$ is endowed with a natural metric, called
the word metric (c.f. \cite{BBI}). For any $x,y\in G,$ the distance between them is defined as the length of the shortest path connecting $x$ and $y,$
$$d^S(x,y):=\inf\{k\in \mathds{N}\mid \exists\  x=x_0\sim x_1\sim\cdots \sim x_k=y\}.$$ It is easy to see that for two generating set $S$ and $S_1$ the metrics $d^S$ and $d^S_1$ are bi-Lipschitz equivalent, i.e. there exist two constants $C_1(S,S_1), C_2(S,S_1)$ depending on $S$ and $S_1$ such that for any $x,y\in G$
$$C_1(S,S_1)d^{S_1}(x,y)\leq d^S(x,y)\leq C_2(S,S_1) d^{S_1}(x,y).$$

Let $B^S_p(n):=\{x\in G| d^S(p,x)\leq n\}$ denote the closed geodesic
ball of radius $n$ centered at $p\in G$ on the Cayley graph $(G,S),$
and $|B_p^{S}(n)|:=\#B_p^S(n)$ the volume (cardinality) of the ball
$B_p^S(n).$ By the group structure, it is obvious that
$|B_p^S(n)|=|B_q^S(n)|,$ for any $p,q\in G.$ The growth function of
$(G,S)$ is defined as $\beta_S(n):=|B_{e}^S(n)|$ where $e$ is the unit
element of $G$. A group $G$ is called of polynomial volume growth if
there exists a finite generating set $S$ such that $\beta_S(n)\leq
Cn^D$ for some $C,D>0$ and any $n\geq1$. It is easy to check that this
definition is independent of the choice of the generating set $S$ by
the bi-Lipschitz equivalence. Thus, the polynomial volume growth is indeed a property of the group $G$.

Gromov \cite{G} proved a celebrated structure theorem for groups of
polynomial volume growth: every finitely generated group of
polynomial volume growth is virtually nilpotent, i.e. it has a
nilpotent subgroup of finite index. Moreover Van den Dries and
Wilkie \cite{VW} showed that it suffices to get one scale polynomial
volume control, that is, if
$$\liminf_{n\rightarrow\infty}\frac{\log\beta_S(n)}{\log n}<\infty,$$ i.e. there exists a subsequence $\{n_i\}_{i=1}^{\infty}$ of $\{n\}$ such that for any $i\in \mathds{N}$
$$\beta_S(n_i)\leq C_1 n_i^D,$$ then $\beta_S(n)\leq C_2 n^D$ for any $n\geq1.$
By Gromov's theorem and Bass' volume growth property \cite{Ba} for nilpotent groups, we have for any Cayley graph $(G,S)$ of polynomial growth
\begin{equation}\label{BAV2}
C_1(S) n^D\leq \beta_S(n)\leq C_2(S)n^D,
\end{equation} for some $D\in \mathds{N}$ and any $n\geq1,$ where $D$ is called the homogenous dimension of $G$. Then it is easy to show the volume doubling property
\begin{equation}\label{VDP1}
|B^S_p(2n)|\leq C(S)|B^S_p(n)|,
\end{equation} for any $p\in G$ and $n\geq1.$

Moreover, Pansu \cite{P} proved the more delicate volume growth property for the Cayley graph $(G,S)$ of polynomial volume growth that
for some $D\in\mathds{N}$ the limit exists
\begin{equation}\label{PANS}
\lim_{n\rightarrow\infty}\frac{\beta_S(n)}{n^D}=C(S)<\infty.
\end{equation}
The following lemma, called the relative volume comparison in the large, is a direct consequence of Pansu's result.
\begin{lemma}\label{RVCAI}
Let $(G,S)$ be a Cayley graph of a group of polynomial volume growth. Then for any $\theta\ll1$ there exists $R_0(\theta,S)$ such that
\begin{equation}\label{RVC}
\frac{|B_p^S(R)|}{|B_p^S(r)|}\leq (1+\theta)\left(\frac{R}{r}\right)^D
\end{equation} for any $p\in G,$ $R\geq r\geq R_0(\theta,S)$ where $D$ is the homogeneous dimension of $G$.
\end{lemma}
\begin{proof}
By Pansu's result \eqref{PANS}, for any $\delta\ll1,$ there exists $R_0(\delta,S)$ such that
$$\left|\frac{B_p^S(n)}{Cn^D}-1\right|<\delta,$$ for any $p\in G,$ $n\geq R_0(\delta,S).$ We denote by $[R]$ the largest integer less than or equal to $R\in\mathds{R}.$ Hence for $R\geq r\geq R_0(\delta,S)$
\begin{eqnarray*}
\frac{|B_p^S(R)|}{|B_p^S(r)|}&\leq&\frac{(1+\delta) C[R]^D}{(1-\delta)C[r]^D}=\frac{1+\delta}{1-\delta}\left(\frac{R}{r}\right)^D\left(\frac{[R]r}{[r]R}\right)^D \\
&\leq&\frac{1+\delta}{1-\delta}\left(\frac{r}{[r]}\right)^D\left(\frac{R}{r}\right)^D\\
&\leq&\frac{1+\delta}{1-\delta}\left(\frac{R_0}{R_0-1}\right)^D\left(\frac{R}{r}\right)^D\\
&=&(1+\theta)\left(\frac{R}{r}\right)^D,
\end{eqnarray*}
where $\theta=\frac{1+\delta}{1-\delta}(\frac{R_0}{R_0-1})^D-1.$ For
$\delta\ll1, R_0(\delta,S)\gg1,$ we have $\theta\ll1$ which proves
the lemma.
\end{proof}

By this lemma, it is easy to show the weak relative volume comparison that
\begin{equation}\label{RV1}
\frac{|B_p^S(R)|}{|B_p^S(r)|}\leq C(S)\left(\frac{R}{r}\right)^D
\end{equation} for any $p\in G,$ $R\geq r\geq 1,$ where $C(S)$ may not be close to $1.$

Kleiner \cite{Kl} proved the Poincar\'e inequality for any Cayley graph $(G,S)$ that there exists a constant $C(S)$ such that for any function $u$ defined on $B_p^S(3n),$ any $p\in G$ and $n\geq1,$
\begin{equation*}\sum_{x\in B_p^S(n)}(u(x)-\bar{u})^2\leq C(S)n^2\frac{|B_p^S(2n)|}{|B_p^S(n)|}\sum_{x,y\in B_p^S(3n);x\sim y}(u(x)-u(y))^2,\end{equation*} where $\bar{u}=\frac{1}{|B_p^S(n)|}\sum_{x\in B_p^S(n)}u(x).$ By the volume doubling property \eqref{VDP1}, we obtain the uniform Poincar\'e inequality for the Cayley graph $(G,S)$ of polynomial volume growth.
\begin{lemma} Let $(G,S)$ be a Cayley graph of a group of polynomial volume growth. Then there exists a constant $C_1(S)$ such that for any function $u$ defined on $B_p(3n),$ any $p\in G$ and $n\geq1,$
\begin{equation}\label{PI1}\sum_{x\in B_p^S(n)}(u(x)-\bar{u})^2\leq C_1(S)n^2\sum_{x,y\in B_p^S(3n);x\sim y}(u(x)-u(y))^2,\end{equation} where $\bar{u}=\frac{1}{|B_p^S(n)|}\sum_{x\in B_p^S(n)}u(x).$
\end{lemma}

For any subset $\Omega\subset G,$ we denote $d^S(x,\Omega):=\inf\{d^S(x,y)\mid y\in \Omega\}$ for any $x\in G,$ $\partial\Omega:=\{z\in G\mid d^S(z,\Omega)=1\},$ and $\bar{\Omega}:=\Omega\cup\partial\Omega.$ For any function $u:\bar{\Omega}\rightarrow \mathds{R},$ the discrete Laplacian operator is defined on $\Omega$ as $(x\in \Omega)$
$$L^Su(x)=\sum_{y\sim x}(u(y)-u(x)).$$ The function $f$ is called harmonic (subharmonic) on $\Omega$ if $L^Su(x)=0\ (\geq0),$ for any $x\in\Omega.$
Let $H^d(G,S):=\{u:G\rightarrow \mathds{R}\mid\ L^S u=0, |u|(x)\leq C(d^S(p,x)+1)^d\}$ denote the space of polynomial growth harmonic functions of growth rate less than or equal to $d$.

Let $X=(V,E)$ be a graph with the set of vertices, $V$, and the set of edges, $E$. For any $x,y\in V,$ they are called neighbors (denoted by $x\sim y$) if $xy\in E.$ The degree of a vertex $x$ is defined as $\deg x:=\#\{y\in V| y\sim x\}.$ For simplicity, we only consider locally finite graphs (i.e. $\deg x<\infty,$ for any $x\in V$) without self-loops and multiedges. A graph $X$ with bounded geometry means a connected graph satisfying $\deg x\leq \Delta$ for all $x\in V$ and some $\Delta>0,$ that is, it has uniformly bounded degree (see Woess \cite{Wo}). For any graph $X,$ there is a natural metric structure $d^X(x,y):=\inf\{k\in \mathds{N}\mid \exists\  x=x_0\sim x_1\sim\cdots \sim x_k=y\},$ where $x,y\in V.$ The closed geodesic ball is denoted by $B^X_p(n):=\{x\in V|d^X(x,p)\leq n\}$ and the volume by $|B_p^X(n)|:=\#B_p^X(n).$ Note that for graphs with bounded geometry it is equivalent to the usual definition in graph theory ($|B_p^X(n)|:=\sum_{x\in B_p^X(n)}\deg x,$ see \cite{Wo}).

 We recall the definition of rough isometries between metric spaces, also called quasi-isometries (see \cite{Gro,BBI,Wo}). For a metric space $(X,d^X)$ and some subset $\Omega\subset X,$ the distance function to $\Omega$ is defined as $d^X(x,\Omega):=\inf\{d^X(x,y)|y\in \Omega\}.$
\begin{definition}\label{DEFR}
Let $(X,d^X), (Y,d^Y)$ be two metric spaces. A rough isometry is a mapping $\phi: X\rightarrow Y$ such that
$$a^{-1}d^X(x,y)-b\leq d^Y(\phi(x),\phi(y))\leq a d^X(x,y)+b,$$ for all $x,y\in X,$ and
$$d^Y(z,\phi(X))\leq b,$$ for any $z\in Y,$ where $a\geq 1,$ $b\geq0.$ It is called an $(a,b)$-rough isometry.
\end{definition}

From a rough isometry $\phi: X\rightarrow Y,$ we can construct a rough inverse $\psi: Y\rightarrow X.$ For any $y\in Y$ we choose $x\in X$ such that $d^Y(y, \phi(x))\leq b$ and set
\begin{equation}\label{quai}\psi(y)=x.\end{equation} Then $\psi$ is an $(a,3ab)$-rough isometry if $\phi$ is an $(a,b)$-rough isometry. It is obvious that the composition of two rough isometries is again a rough isometry. Hence, to be roughly isometric is an equivalence relation between metric spaces.

In this paper, we only consider rough isometries between metric spaces of graphs with bounded geometry. It is well known that the volume doubling property and the Poincar\'e inequality are roughly isometric invariants (see \cite{CS,Kl}). Hence by the volume doubling property \eqref{VDP1} and the Poincar\'e inequality \eqref{PI1} for Cayley graphs of groups of polynomial growth, we have
\begin{lemma}\label{LMVD1}
Let $(X,d)$ be a graph with bounded geometry roughly isometric to a Cayley graph $(G,S,d^S)$ of a group of polynomial volume growth. Then
\begin{equation}\label{OXVD1}
|B_p^X(2R)|\leq C|B_p^X(R)|,
\end{equation} for any $p\in X, R\geq 1.$
\end{lemma}
\begin{lemma}\label{LMPI1}
Let $(X,d)$ be a graph with bounded geometry roughly isometric to a Cayley graph $(G,S,d^S)$ of a group of polynomial volume growth. Then there exist constants $C_1(\Delta,a,b,S)$ and $C_2(\Delta,a,b,S)$ such that for any $p\in X, R\geq 1$ and any function $u$ defined on $B_p^X(C_1R)$ we have
\begin{equation}\label{RIPI1}\sum_{x\in B_p^X(R)}(u(x)-\bar{u})^2\leq C_2R^2\sum_{x,y\in B_p^X(C_1R);x\sim y}(u(x)-u(y))^2,\end{equation} where $\bar{u}=\frac{1}{|B_p^X(R)|}\sum_{x\in B_p^X(R)}u(x).$
\end{lemma}

In the sequel, for the Cayley graph $(G,S)$ we omit the dependence of
the generating set $S$, like $B_p(n):=B^S_p(n)$ if there is no danger
of confusion; in other cases we denote it by $B^G_p(n)$ if we need to
emphasize  the difference from $B_p^X(n)$ in the graph $X$. In addition, the sum over a geodesic ball is denoted by the integration with respect to the counting measure $\int_{B_p^S(n)}u^2:=\sum_{x\in B^S_p(n)}u^2(x).$

\section{Mean value inequality and optimal dimension estimate}
In this section, we obtain the mean value inequality by the volume doubling property \eqref{VDP1} and the Poincar\'e inequality \eqref{PI1} on Cayley graphs of groups of polynomial volume growth. Then we use the relative volume comparison \eqref{RVC} and the mean value inequality to get the optimal dimension estimate for $H^d(G,S)$ by Li's argument \cite{L1,L2}.

Delmotte \cite{D2} and Holopainen-Soardi \cite{HS} independently carried out the Moser iteration on graphs satisfying the volume doubling property and the Poincar\'e inequality which implies the Harnack inequality.
\begin{lemma}[Harnack inequality]
Let $(G,S)$ be a Cayley graph of a group of polynomial volume growth. Then there exist constants $C_1(S), C_2(S)$ such that for any $p\in G, n\geq 1,$ any positive harmonic function $u$ on $B_p(C_1n)$ we have
$$\max_{B_p(n)}u\leq C_2\min_{B_p(n)}u.$$
\end{lemma}

The mean value inequality is one part of the Moser iteration (see \cite{D2,HS,CoG}).
\begin{lemma}[Mean value inequality]\label{MVI1}
Let $(G,S)$ be a Cayley graph of a group of polynomial volume growth. Then there exists a constant $C_1(S)$ such that for any $p\in G, R\geq1,$ any harmonic function $u$ on $B_p(R)$ we have
\begin{equation}\label{MVI}
u^2(p)\leq \frac{C_1}{|B_p(R)|}\sum_{x\in B_p(R)}u^2(x).
\end{equation}
\end{lemma}

The dimension estimate follows from Li's argument in \cite{L1} (see also \cite{L2,D1,Hu2,HJ}). We need some lemmas.
\begin{lemma}
For any finite dimensional subspace $K\subset H^d(G,S),$ there exists a constant $R_1(K)$ depending on $K$ such that for any $R\geq R_1(K)$
$$A_R(u,v):=\int_{B_p(R)}uv:=\sum_{x\in B_p(R)}u(x)v(x)$$ is an inner product on $K$.
\end{lemma}
\begin{proof} A contradiction argument (see \cite{Hu2}).
\end{proof}

\begin{lemma}\label{LM1}
Let $(G,S)$ be a Cayley graph of a group of polynomial volume growth with homogeneous dimension $D$, $K$ be a $k$-dimensional subspace of $H^d(G,S)$. Given $\beta>1,\delta>0$ for any $R_1'\geq R_1(K)$ there exists $R>R_1'$ such that if $\{u_i\}_{i=1}^k$ is an orthonormal basis of $K$ with respect to the inner product $A_{\beta R}$, then
$$\sum_{i=1}^k A_R(u_i,u_i)\geq k\beta^{-(2d+D+\delta)}.$$
\end{lemma}
\begin{proof}
The lemma follows from the volume growth property \eqref{BAV2} and the linear algebra (see \cite{L1,L2,D1,Hu2}).
\end{proof}

The next lemma follows from the mean value inequality \eqref{MVI} of
Lemma \ref{MVI1} for harmonic functions.
\begin{lemma}\label{LM2}
Let $(G,S)$ be a Cayley graph of a group of polynomial volume growth with homogeneous dimension $D$, $K$ be a $k$-dimensional subspace of $H^d(G,S).$ Then for any fixed $0<\epsilon<\frac{1}{2}$ there exist constants $C(S)$ and $R_2(\epsilon,S)$ such that for any $R\geq R_2(\epsilon,S)$ and any basis $\{u_i\}_{i=1}^k$ of $K$ we have
$$\sum_{i=1}^k A_R(u_i,u_i)\leq C\epsilon^{-(D-1)}\sup_{u\in<A,U>}\int_{B_p(1+\epsilon)R}u^2,$$ where $<A,U>:=\{w=\sum_{i=1}^ka_iu_i|\sum_{i=1}^ka_i^2=1\}.$
\end{lemma}
\begin{proof} For any $x\in B_R(p),$ we set $K_x=\{u\in K: u(x)=0\}.$ It
  is easy to see that $\dim K/K_x\leq1.$ Hence there exists an orthonormal linear transformation $\varphi:K\rightarrow K,$ which maps $\{u_i\}_{i=1}^k$ to $\{v_i\}_{i=1}^k$ such that $v_i\in K_x,$ for $i\geq2.$ For any $x\in B_p(R),$ since $\epsilon R\geq \epsilon R_2\geq 1$ by choosing $R_2\geq\frac{1}{\epsilon},$ then $(1+\epsilon)R-r(x)\geq 1$ for $r(x)=d(p,x).$
Hence the mean value inequality (\ref{MVI}) implies that for any $x\in B_p(R)$
\begin{eqnarray}\label{SOM1}
\sum_{i=1}^ku_i^2(x)&=&\sum_{i=1}^kv_i^2(x)=v_1^2(x)\nonumber\\
&\leq&C(S)\left|B_x((1+\epsilon)R-r(x))\right|^{-1}\int_{B_x((1+\epsilon)R-r(x))}v_1^2\nonumber\\
&\leq&C(S)\left|B_x((1+\epsilon)R-r(x))\right|^{-1}\sup_{u\in
<A,U>}\int_{B_p(1+\epsilon)R}u^2.
\end{eqnarray}
For simplicity, we denote $V_p(t):=|B_p(t)|.$
By the weak relative volume comparison (\ref{RV1}), we have
\begin{eqnarray*}V_x((1+\epsilon)R-r(x))&\geq&\frac{1}{C(S)}\left(\frac{(1+\epsilon)R-r(x)}{2R}\right)^DV_x(2R)\\
&\geq&\frac{1}{C(S)}\left(\frac{(1+\epsilon)R-r(x)}{2R}\right)^DV_p(R).\end{eqnarray*}
Hence, substituting it into (\ref{SOM1}) and integrating over
$B_p(R)$, we have
\begin{equation}\label{ode1}
\sum_{i=1}^k\int_{B_p(R)}u_i^2\leq \frac{C(S)}{V_p(R)}\sup_{u\in
\langle
A,U\rangle}\int_{B_p(1+\epsilon)R}u^2\int_{B_p(R)}(1+\epsilon-R^{-1}r(x))^{-D}dx
\end{equation}
Define $f(t)=(1+\epsilon-R^{-1}t)^{-D}.$ It suffices to bound the term
\begin{eqnarray}\label{ode2}
\int_{B_p(R)}f(r(x))dx&=&\int_{B_p(R)\backslash B_p(\frac{R}{2})}f+\int_{B_p(\frac{R}{2})}f\nonumber\\
&\leq&\int_{B_p(R)\backslash B_p(\frac{R}{2})}f+CV_p(\frac{R}{2})\nonumber\\
&\leq&\int_{B_p(R)\backslash B_p(\frac{R}{2})}f+CV_p(R)\nonumber\\
&=& (\ast)+CV_p(R)
\end{eqnarray}

We denote by $[R]$ the largest integer less than or equal to $R.$ Then for $[\frac{R}{2}]\geq R_0(\theta,S)$ (e.g. $R\geq 3R_0(\theta,S)$ is sufficient) where $R_0(\theta,S)$ is in Lemma \ref{RVCAI},
\begin{eqnarray}\label{ode3}
(\ast)&\leq&\sum_{[\frac{R}{2}]\leq i\leq [R]}f(i)(V_p(i)-V_p(i-1))\nonumber\\
&=&\sum_{[\frac{R}{2}]\leq i\leq [R]-1}(f(i)-f(i+1))V_p(i)+V_p([R])f([R])-V_p([\frac{R}{2}]-1)f([\frac{R}{2}])\nonumber\\
&\leq& \frac{V_p(R)}{R^D}(1+\theta)^{-1}\sum_{[\frac{R}{2}]\leq i\leq [R]-1}(f(i)-f(i+1))i^D+V_p(R)f([R])\nonumber\\
&\leq& \frac{V_p(R)(1+\theta)^{-1}}{R^D}\Big[\sum_{[\frac{R}{2}]+1\leq i\leq [R]-1} f(i)(i^D-(i-1)^D)+f([\frac{R}{2}])[\frac{R}{2}]^D\nonumber\\
&&-f([R])([R]-1)^D\Big]+V_p(R)f([R])\nonumber\\
&\leq&f([R])V_p(R)\left[1-(1+\theta)^{-1}\left(\frac{[R]-1}{R}\right)^D\right]\nonumber\\
&&+\frac{V_p(R)(1+\theta)^{-1}}{R^D}\sum_{[\frac{R}{2}]+1\leq i\leq [R]-1} f(i)(i^D-(i-1)^D)+CV_p(R)\nonumber\\
&=&I+II+CV_p(R)
\end{eqnarray}
For fixed $0<\epsilon<\frac{1}{2},$ there exist $\theta=C(\epsilon)\ll1$ and $R_1(\epsilon)\gg1$ such that for any $R\geq R_2\geq R_1(\epsilon),$
$$\left|1-(1+\theta)^{-1}\left(\frac{[R]-1}{R}\right)^D\right|\leq \epsilon,$$ hence
\begin{equation}\label{ode4}I\leq CV_p(R)\epsilon^{-D}\cdot\epsilon=CV_p(R)\epsilon^{-(D-1)}.\end{equation}

For the second term,
\begin{eqnarray}\label{ode5}
II&\leq& C(D)(1+\theta)^{-1}\frac{V_p(R)}{R^D}\sum_{[\frac{R}{2}]+1\leq i\leq [R]-1} f(i)i^{D-1}\nonumber\\
&\leq& C\frac{V_p(R)}{R}\sum_{[\frac{R}{2}]+1\leq i\leq [R]-1}f(i)\nonumber\\
&\leq& C\frac{V_p(R)}{R}\int_{[\frac{R}{2}]}^Rf(t)dt\nonumber\\
&\leq& C\epsilon^{-(D-1)}V_p(R).
\end{eqnarray}

Combining the estimates of \eqref{ode2} \eqref{ode3} \eqref{ode4} and \eqref{ode5}, we obtain that for $\theta=C(\epsilon),$ any $R\geq R_2(\epsilon,S)=\max\{\frac{1}{\epsilon},3R_0(\theta,S),R_1(\epsilon)\}$
\begin{eqnarray}\label{ode6}
\int_{B_p(R)}f(r(x))dx&\leq&C(\epsilon^{-(D-1)}+C)V_p(R)\nonumber\\
&\leq& C\epsilon^{-(D-1)}V_p(R).
\end{eqnarray}
The lemma follows from \eqref{ode1} and \eqref{ode6}.

\end{proof}

\begin{proof}[Proof of Theorem \ref{MTAL}] For any $k$-dimensional subspace $K\subset H^d(G,S),$ we set $\beta=1+\epsilon,$ for fixed small $0<\epsilon<\frac{1}{2}$. By Lemma \ref{LM1}, there exists infinitely many $R>R_1(K)$ such that for any orthonormal basis $\{u_i\}_{i=1}^k$ of $K$ with respect to $A_{(1+\epsilon) R},$ we have
$$\sum_{i=1}^k A_R(u_i,u_i)\geq k(1+\epsilon)^{-(2d+D+\delta)}.$$
Lemma \ref{LM2} implies that for sufficiently large $R$ $$\sum_{i=1}^k A_R(u_i,u_i)\leq C(S)\epsilon^{-(D-1)}.$$ Setting $\epsilon=\frac{1}{2d},$ and letting $\delta\rightarrow 0,$ we obtain
\begin{equation}k\leq C(S)\left(\frac{1}{2d}\right)^{-(D-1)}\left(1+\frac{1}{2d}\right)^{2d+D+\delta}\leq C(S)d^{D-1}.\end{equation} which proves the theorem.
\end{proof}

\section{Rough isometries and optimal dimension estimate}
In this section we obtain the optimal dimension estimate for
polynomial growth harmonic functions on graphs with bounded geometry
roughly isometric to Cayley graphs of groups of polynomial volume
growth. The strategy is similar to that in \cite{HJ}. Let $(X,d^X)$ be
a graph with bounded geometry roughly isometric to a Cayley graph
$(G,S,d^S)$ of a group of polynomial volume growth. Since the rough
isometry does not preserve the optimal volume growth condition
\eqref{RVC}, it is hard to get the optimal dimension estimate by the
argument on the graph $X$. Instead of doing that, we construct
functions on $(G,S)$ from harmonic functions on $X$ which are not
harmonic on $(G,S)$ but satisfy the mean value inequality in the large
\eqref{MVF1}. By the same arguments as in Section 2, we obtain the optimal dimension estimate for functions on $(G,S)$ which implies the dimension estimate on $X.$

In order to preserve the mean value inequality under our later construction, we need the following lemma which says that we can always construct an injective rough isometry $\phi':X\rightarrow G'$ from a rough isometry to a Cayley graph $(G,S)$ of a group of polynomial volume growth $\phi:X\rightarrow G,$ where $G'$ is also of polynomial volume growth with same homogenous dimension as $G.$
\begin{lemma}
Let $(X,d^X)$ be a graph with bounded geometry roughly isometric to a Cayley graph $(G,S,d^S)$ of a group of polynomial volume growth, $\phi:X\rightarrow G$ be the rough isometry. Then there exist a group $G'$ of polynomial volume growth with same homogenous dimension as $G$ and an injective rough isometry $\phi':X\rightarrow G'$ which is constructed from $\phi.$
\end{lemma}
\begin{proof}
Since $X$ is a graph with bounded geometry, $\deg x\leq \Delta$ for all $x\in X.$ Assume that $\phi:X\rightarrow G$ is an $(a,b)$-rough isometry which is generally not injective. For any $y\in \phi(X),$ by rough isometry $\diam\{\phi^{-1}(y)\}\leq ab.$ Hence by bounded geometry of $X,$
\begin{equation}\label{Boundgeom1}
\#\{\phi^{-1}(y)\}\leq \Delta^{[ab]+1}=:q(\Delta,a,b).
\end{equation}

Define a group $G':=G\times\mathds{Z}_q$ where $\mathds{Z}_q=\{0,1,\cdots,q-1\}$ is a finite cyclic group and $\times$ is the direct product of groups (see \cite{La}). Let $S':=S\times\{0\}\cup(e,\pm1),$ where $e$ is the unit element of $G.$ It is easy to see that $S'$ is a generating set of $G'$ and $(G',S')$ and $(G,S)$ are of polynomial volume growth with same homogenous dimension. By \eqref{Boundgeom1}, it is straightforward to define an injective rough isometry $\phi':X\rightarrow G'=G\times\mathds{Z}_q$ such that $$\pi_G\circ\phi'=\phi,$$ where $\pi_G$ is a standard projection map $\pi_G:G'\rightarrow G.$ Actually it suffices to separate the image under $\phi'$ of $\phi^{-1}(y)$ into $y\times \mathds{Z}_q$ for any $y\in\phi(X).$
\end{proof}

By this lemma, we always assume that (to keep the notations) there is an injective $(a,b)$-rough isometry $\phi: X\rightarrow G$ to a Cayley graph of a group of polynomial volume growth. The Laplacian operator $L^X$ is defined for any function $u$ on $X$ as
$$L^Xu(x):=\sum_{y\sim x}(u(y)-u(x)).$$ We denote by $H^d(X):=\{u:X\rightarrow \mathds{R}| L^Xu=0, |u|(x)\leq C(d^X(x,p)+1)^d\}$ the space of polynomial growth harmonic functions on $X$ with growth rate less than or equal to $d,$ by $P^d(G):=\{u:G\rightarrow \mathds{R}|\ |u|(x)\leq C(d^S(x,p_0)+1)^d\}$ the space of polynomial growth functions on $G$ with growth rate less than or equal to $d,$ where $p\in X$ and $p_0\in G$ are fixed. For a fixed injective rough isometry $\phi: X\rightarrow G,$ we construct a function $\widetilde{u}$ on $G$ from a function $u:X\rightarrow \mathds{R}$ as follows. For any $y\in \phi(X),$ let $\widetilde{u}(y):=u(\phi^{-1}(y)).$ For $y\in G\backslash\phi(X),$ let $W_y:=\phi^{-1}(\phi(X)\cap B_b^G(y))\subset X$ and \begin{equation}\label{dfi1}\widetilde{u}(y):=\frac{1}{|W_y|}\sum_{x\in W_y}u(x),\end{equation} where $|W_y|:=\# W_y.$ We define
$$E:H^d(X)\rightarrow P^d(G)$$
$$u\mapsto Eu:=\widetilde{u}.$$ It is easy to see that $E$ is an injective linear operator. Hence $$\dim H^d(X)=\dim EH^d(X).$$ To estimate $\dim H^d(X),$ it suffices to bound the dimension of $EH^d(X).$ The advantage of estimating $\dim EH^d(X)$ on $G$ is the good volume growth property \eqref{RVC}. Next we shall prove the crucial property (mean value inequality) of $Eu=\widetilde{u}$ for any harmonic function $u$ on $X.$

As in Lemma \ref{MVI1}, the mean value inequality for harmonic
functions on $X$ follows from the volume doubling property
\eqref{OXVD1} and the Poincar\'e inequality \eqref{RIPI1} (see
\cite{D2,HS,CoG}).
\begin{lemma}\label{MVFXX}
Let $(X,d^X)$ be a graph with bounded geometry roughly isometric to a Cayley graph $(G,S,d^S)$ of a group of polynomial volume growth. Then there exists a constant $C_1$ such that for any $p\in X, R\geq1,$ any harmonic function $u$ on $B_p^X(R)$ we have
\begin{equation}\label{MVIOX1}
u^2(p)\leq \frac{C_1}{|B_p^X(R)|}\sum_{x\in B_p^X(R)}u^2(x).
\end{equation}
\end{lemma}

The proof of the following lemma is essentially the same as that of Lemma \ref{LMVD1}.
\begin{lemma}
Let $\phi:X\rightarrow G$ be an injective $(a,b)$-rough isometry to a Cayley graph $(G,S)$ of a group of polynomial volume growth. Then there exist constants $C_1<1, C_2$ and $R_1$ such that for any $y\in G,p\in X$ satisfying $d^S(\phi(p),y)\leq b$ and $R\geq R_1$ we have
\begin{equation}\label{VL1}
|B_p^X(R)|\geq C_2|B_y^G(C_1R)|.
\end{equation}
\end{lemma}
\begin{proof}
For fixed $y\in G, p\in X$ satisfying $d^S(\phi(p),y)\leq b,$ by \eqref{quai} we can construct an inverse $(a,3ab)$-rough isometry $\psi=\psi_{y,p}:G\rightarrow X$ from $\phi$ such that $\psi(y)=p.$ We claim that $\psi (B_y^G(C_1 R))\subset B_p^X(R)$ for some $C_1<1$ and any $R\geq R_1.$ For any $q\in \psi(B_y^G(C_1R))$ ($C_1$ to be chosen), $q=\psi(z)$ for some $z\in B_y^G(C_1R).$ Then
\begin{eqnarray*}
d^X(p,q)&=& d^X(\psi(y),\psi(z))\leq ad^S(y,z)+3ab\\
&\leq& aC_1R+3ab\leq R,
\end{eqnarray*} if we choose $C_1=\frac{1}{2a}, R_1=6ab$ which yields the claim.

By the bounded geometry of $G,$ for any $s\in X,$ $\#\{\psi^{-1}(s)\}\leq C.$ Hence
$$|B_p^X(R)|\geq\#\psi(B_y^G(C_1R))\geq\frac{1}{C}|B_y^G(C_1R)|.$$ The lemma follows.
\end{proof}

The mean value inequality for $Eu=\widetilde{u}$ follows from the previous two lemmas and the injectivity of the rough isometry $\phi.$
\begin{theorem}[Mean value inequality in the large]\label{MVIIL}
Let $\phi:X\rightarrow G$ be an injective $(a,b)$-rough isometry to a Cayley graph of a group of polynomial volume growth. Then there exist constants $C, R_0$ such that
for any harmonic function $u$ on $X,$ any $y\in G$ and $R\geq R_0$ we have
\begin{equation}\label{MVF1}
\widetilde{u}^2(y)\leq \frac{C}{|B_y^G(R)|}\sum_{x\in B_y^G(R)}\widetilde{u}^2(x).
\end{equation}
\end{theorem}
\begin{proof}
By the group structure of $G,$ for any $y\in G,$ $|B^G_y(b)|=\beta(b)\leq C(b).$ The bounded geometry of $X$ implies that $\#\phi^{-1}(B^G_y(b))\leq C(\Delta,a,b).$
Then by the definition of $\widetilde{u}$ \eqref{dfi1}, there exists $p\in W_y$ (e.g. the maximum point for $u^2$ in $W_y$) such that for any $R\geq R_1$
\begin{eqnarray*}
\widetilde{u}^2(y)&\leq& C(\Delta,a,b)u^2(p)\\
&\leq&\frac{C}{|B_p^X(R)|}\sum_{x\in B_p^X(R)}u^2(x) \ \ \ \ \ \  by\ \eqref{MVIOX1}\\
&\leq&\frac{C}{|B_y^G(C_1R)|}\sum_{x\in B_p^X(R)}u^2(x) \ \ \  \ by\  \eqref{VL1}.
\end{eqnarray*}
For any $z\in \phi(B_p^X(R)),$ $z=\phi(q)$ where $q\in B^X_p(R).$ Then
\begin{eqnarray*}
d^G(z,y)&\leq&d^G(z,\phi(p))+d^G(\phi(p),y)\leq ad^X(p,q)+b+b\\
&\leq& a R+2b\leq CR,
\end{eqnarray*} if we choose $C=2a, R\geq \frac{2b}{a}$ which implies that
$\phi(B_p^X(R))\subset B_y^G(CR).$ Hence, for $R\geq R_2=\max\{1,R_1,\frac{2b}{a}\}$
\begin{eqnarray*}
\widetilde{u}^2(y)&\leq&\frac{C}{|B_y^G(C_1R)|}\sum_{x\in B_p^X(R)}u^2(x)\\
&\leq&\frac{C}{|B_y^G(C_1R)|}\sum_{w\in B_{CR}^G(y)\cap \phi(X)}\widetilde{u}^2(w)\\
&\leq&\frac{C}{|B_y^G(C_1R)|}\sum_{w\in B_{CR}^G(y)}\widetilde{u}^2(w)\\
&\leq&\frac{C}{|B_y^G(CR)|}\sum_{w\in B_{CR}^G(y)}\widetilde{u}^2(w),
\end{eqnarray*} where we use the injectivity of $\phi$ and the volume doubling property \eqref{VDP1} on $G.$ The theorem follows from choosing $(R_0=CR_2).$
\end{proof}

\begin{proof}[Proof of Theorem \ref{MT2}] Since we have obtained the
  mean value inequality in the large \eqref{MVF1}, combining with the
  relative volume comparison in the large \eqref{RVC}, we prove the
  theorem by the same argument as in  Section 2 (see \cite{HJ}).
\end{proof}

As an application, we obtain the dimension estimate for the quasi-transitive graphs (see Woess \cite{Wo}). Let $(X,d^X)$ be a locally finite, connected graph. An automorphism of $X$ is a self-isometry of $(X,d^X).$ We denote by $Aut(X)$ the set of automorphisms of $X.$ The graph $X$ is called vertex transitive if $Aut(X)$ acts transitively on $X,$ i.e. the factor (quotient) graph $X\slash Aut(X)$ has only one orbit. It is called quasi-transitive if $Aut(X)$ acts with finitely many orbits. It is easy to see that transitive and quasi-transitive graphs possess bounded geometry. We recall a theorem in \cite{Wo} (Theorem 5.11).
\begin{lemma}\label{LH11}
Let $(X,d^X)$ be a quasi-transitive graph whose growth function satisfies $|B_p^X(n)|\leq Cn^A$ for infinitely many $n$ and some $A>0.$ Then $X$ is roughly isometric to a Cayley graph of some finitely generated nilpotent group with homogenous dimension $D.$
\end{lemma}

By Theorem \ref{MT2} and Lemma \ref{LH11}, we obtain the following corollary.
\begin{corollary}\label{COAC}
Let $(X,d^X)$ be a quasi-transitive graph whose growth function satisfies $|B_p^X(n)|\leq Cn^A$ for infinitely many $n.$ Then
$$\dim H^d(X)\leq Cd^{D-1},$$ for $d\geq1$ where $D$ is the homogenous dimension of the nilpotent group in Lemma \ref{LH11}.
\end{corollary}

\end{document}